\documentclass{amsart}

\usepackage{amssymb}
\usepackage{amsmath}
\usepackage{amsthm}
\usepackage{amsbsy}
\usepackage{bm}
\usepackage[dvips]{graphicx}

\theoremstyle{plain}
\newtheorem{theorem}{Theorem}[section]
\newtheorem{lemma}[theorem]{Lemma}

\theoremstyle{definition}
\newtheorem{definition}[theorem]{Definition}

\theoremstyle{remark}
\newtheorem{remark}[theorem]{Remark}
\newtheorem{claim}{Claim}

\newtheorem*{acknowledgements}{Acknowledgements}

\newcommand{\R}{\mathbb{R}}

\newcommand{\del}{\partial}

\begin{document}

\title{Euler characteristic and quadrilaterals of normal surfaces}

\author{Tejas Kalelkar}

\address{   Stat-Math Unit\\
        Indian Statistical Institute\\
        Bangalore, India}

\email{tejas@isibang.ac.in}

\date{\today}

\subjclass{Primary 57Q35 ; Secondary 57M99} 

\begin{abstract}
Let $M$ be a compact 3-manifold with a triangulation $\tau$. We give an inequality relating the Euler characteristic of a
surface $F$ normally embedded in $M$ with the number of normal quadrilaterals in $F$. This gives a relation
between a topological invariant of the surface and a quantity derived from its combinatorial description. Secondly, we obtain an inequality relating the number of normal triangles and normal quadrilaterals of $F$, that depends on the maximum number of tetrahedrons that share a vertex in $\tau$.
\end{abstract}

\maketitle

\section{Introduction}
Let $M$ be a compact 3-manifold with a triangulation $\tau$. The triangulation $\tau$ may be a pseudo-triangulation. A pseudo triangulation is a simplicial-complex where each 3-cell is a tetrahedron whose interior is embedded, but the boundary may not injective.A surface $F$
in $M$ is said to be normal if it is properly embedded in $M$, it intersects each simplex
transversally and it intersects each tetrahedron in a disk that can be isotoped, by an isotopy that is
invariant on each simplex, to one of the 7 disk types shown in Figure \ref{Fig1}.

Normal surfaces were first introduced by Kneser\cite{Kn} in 1929
and later developed by Schubert\cite{Schu} and Haken\cite{Ha, Ha2} for use in
an algorithm for recognising the unknot. Haken began the
construction of an algorithm for solving the homeomorphism problem
for a large class of 3-manifolds, the details of which were
completed by Jaco and Oertel\cite{Ja} and Hemion\cite{He} for
manifolds containing an incompressible 2-sided surface. Normal surfaces have been used for a variety of
tasks by Rubinstein\cite{Ru, Ru2}, Jaco and Tollefson\cite{Ja2}
and Jaco, Letscher and Rubinstein\cite{Ja3}. For an extensive
study of normal surface theory we refer to \cite{Ma}.

An embedded normal surface is determined uniquely by the number of triangles
and quadrilaterals of a type within each tetrahedron. These are called the normal coordinates of the normal surface. There exists a system of homogenous linear equations called the matching equations which ensure that normal disks in adjacent
tetrahedrons glue together to form a surface. As quadrilaterals of different types within the same tetrahedron intersect, the normal coordinates of a normal surface have at most one non-zero coordinate corresponding to the different quadrilateral types, in each tetrahedron. A solution
to the matching equations is called admissible if at most one of the
coordinates corresponding to the different quadrilateral types is
nonzero, in each tetrahedron. The non-negative admissible
solutions of the matching equations are in bijection with the set
of embedded (not necessarily connected) normal surfaces. Every such solution is a linear combination of
a finite set of fundamental solutions with non-negative integer
coefficients. The additive structure of the solution space is
geometrically realised as the Haken sum of the corresponding
surfaces. It is this combinatorial description of normal surfaces
that is utilised in solving algorithmic problems.

The goal of this paper is to give a relation between the Euler
characteristic of a normal surface, a topological invariant, and
the number of normal quadrilaterals in its embedding, obtained from its
combinatorial description. Secondly, we get a relation between the
number of normal triangles and normal quadrilaterals.

\begin{figure}
\centering
\includegraphics[width=0.7\textwidth]{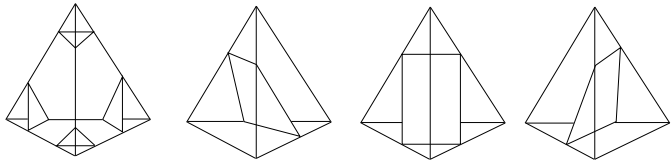}
\caption{The 4 Normal Triangle and 3 Normal Quadrilateral
Types}\label{Fig1}
\end{figure}

\begin{theorem}\label{genus-quad}
Let $F$ be a normal surface in a triangulated 3-manifold $M$. Let $Q$ be the number of normal
quadrilaterals in $F$. Then,
$$ \chi(F) \geq 2 - 7 Q$$
In particular, if $F$ is an oriented, closed and connected normal surface
of genus $g$,
$$ g \leq \frac{7}{2} Q $$
\end{theorem}



\begin{definition}
Let $F$ be a normal surface in $M$. Let $t$ be a normal triangle of $F$ that lies in a
tetrahedron $\Delta$. The triangle $t$ is said to link a vertex $v$ of
$\Delta$ if $t$ separates $\del \Delta$ into two disks such
that the disk containing $v$ has no other vertices of $\Delta$.
Similarly, a normal arc $\alpha$ in a face $\mathcal{F}$ is said to link a
vertex $v$ of $\mathcal{F}$ if the segment containing $v$ in $\del \mathcal{F} - \alpha$ has no other vertices of $\mathcal{F}$.
\end{definition}

\begin{definition}
Let $S(v)$ be the boundary of a small ball neighbourhood of $v$. The sphere $S(v)$ is a normal surface composed of normal triangles linking $v$, one from each normal isotopy class. This is defined to be the vertex linking sphere linking vertex $v$.
\end{definition}

\begin{remark}\label{vertexlinking sphere}
Any closed connected normal surface $S$ in $M$ composed of normal triangles is normally isotopic to a vertex-linking sphere. This is due to the following reasons.

Firstly, all normal triangles in $S$ link the same vertex. If this were not true, there would be normal triangles $t_1$ and $t_2$ in $S$, linking distinct vertices of $\tau$, that intersect in common edges. This is not possible as normal arcs linking different vertices of a face are not normally isotopic. So, let $v$ be the common vertex linked by triangles of $S$. Then $S$ is a cover of $S(v)$, as $S(v)$ is composed of one triangle from each of the normal isotopy classes of triangles linking vertex $v$. As $S(v)$ is a sphere and $S$ is a closed connected cover of $S(v)$, the covering projection map is a homeomorphism. Therefore, there is only one triangle of $S$ in each normal isotopy class of triangles linking $v$ and thus, $S$ is normally isotopic to $S(v)$.
\end{remark}

This remark is the motivation for Theorem \ref{genus-quad}. The remark is not true for
closed connected normal surfaces embedded in 3-complexes. For example, take
a triangulation of a surface $F$ and construct a 3-complex
by taking the join with a common point of all triangles in
$F$. Push $F$ into the interior of this 3-complex so
that $F$ is embedded as a closed normal surface composed
entirely of normal triangles.

We  expect that in Riemannian 3-manifolds, triangles
correspond to positive curvature pieces. This would imply that
when quadrilaterals have a curvature that is bounded below, the
smaller the Euler characteristic, the greater would be the number
of quadrilaterals.

\begin{remark}
An inequality relating the Euler characteristic with the total number of normal disks is easy to obtain. This is because the number of normal discs gives a bound on the number of disjoint closed curves on the surface, up to ispotopy. This, in turn, gives a lower bound on the Euler characteristic.
\end{remark}

Our second theorem, Theorem \ref{tria-quad}, gives a relation between the number of
triangles and the number of quadrilaterals of an embedded normal
surface. Let $A$ be the union of triangles of an
embedded normal surface $F$. Let $v$ be a vertex of the
triangulation $\tau$. Let $N_v$ be the number of tetrahedrons containing
$v$. Let $\sigma$ be a \textit{strongly connected} (defined in Definition \ref{strongly conn defn}) component of the triangles of $A$ linking vertex $v$. By lemma \ref{planar surfaces}, $\sigma$ has at most $N_v$ triangles. Any two such components are `connected' by quadrilaterals. So we get an upper bound on the number of triangles in terms of the number of quadrilaterals and the maximum number of tetrahedrons linking a vertex of $\tau$.

\begin{theorem}\label{tria-quad}
Let $F$ be a normal surface in $M$, no component of which is a vertex-linking sphere. Let $T$ and
$Q$ be the number of normal triangles and normal quadrilaterals of
$F$ respectively. Let $N$ be the maximum number of
tetrahedrons linking a vertex of $\tau$. Then,
$$ T \le 4 N Q $$
\end{theorem}

\section{An inequality relating the Euler characteristic and the number of
quadrilaterals of normal surfaces}

In this section we give a proof of Theorem \ref{genus-quad}. Let $F$ be a normal surface in $M$. Let $Q$ be the
number of normal quadrilaterals in $F$. Let $A$ (respectively $B$) be
the union of normal triangles (respectively normal quadrilaterals) of
$F$. $A$ and $B$ are 2-complexes which are subsets of the surface $F$. Define the boundary of a 2-complex $\sigma$ to be the union of edges of $\sigma$ that do not intersect $int(\sigma)$.

An outline of the proof is as follows. We suitably modify the 2-complexes $A$ and $B$, at
singular points, to get surfaces $A'$ and $B'$ that intersect in a disjoint union of
circles with $A' \cup B' = F$. We then show that the connected components of  $A'$ are planar surfaces. We, therefore, get a lower bound on the Euler
characteristic of $A'$ in terms of the number of components $|\del A'|$ of $\del A'$. The boundaries $\del A'$ and $\del B'$ are equal. We introduce weights and get an inequality relating $|\del B'|$ with the weight $w(\del B)$, of $\del B$. Now as $w(\del B) \leq w(B) = 4Q$ we obtain an upper bound on $ |\del A'|$ in terms of the number of quadrilaterals. Thus, we get a lower bound on the Euler characteristic of $A'$ in terms of the number of quadrilaterals.

For a bound on the Euler characteristic of $B'$, we take an
increasing union of quadrilaterals and use the fact that each new
quadrilateral intersects the union of the preceding stage in a circle or at
most 4 disjoint segments. This gives us a lower bound in terms of the number of quadrilaterals.

As the Euler characteristic of $F$ is the sum of Euler characteristics of $A'$ and $B'$, we get a lower bound for the
Euler characteristic of $F$ in terms of $Q$.

\begin{definition}
Let $D_n$ be the collection of normal discs of $F$. Let $D'_n$ be the collection of
triangles obtained by taking the barycentric subdivision of $D_n$. Then, the regular
neighbourhood of a point $w$ in $F$ is defined as $N(w) =  int ( \cup\{ D'_n : w \in D'_n\} )$.
\end{definition}

\begin{definition}
A point $w$ in $A$ is said to be a singular point if $N(w)$ is not
homeomorphic to an open set in the upper half-plane $\{
(x,y) \in \R^2 | \, y \geq 0 \}$. Let $\omega$ be the set of
singular points of $A$ and $N(\omega) = \cup \{ N(w) : w \in
\omega\}$.
\end{definition}

Let $A'=A - N(\omega)$, $B' = B \cup \overline{N(\omega)}$. Therefore, $A'$ and
$B'$ are compact surfaces with $F = A' \cup B'$ and $A' \cap B'$ a disjoint union of circles. The surface $A'$ is the closure of the interior of $A$, while $B'$ is the closure of a neighbourhood of $B$ in $F$.

\begin{definition}\label{strongly conn defn}
A subcomplex $\sigma$ of $A$ is said to be strongly connected if
for any two triangles $t$ and $t'$ in $\sigma$, there exists a
sequence of triangles $\{ t_i \}_{i=1}^n$ with $t = t_1$
and $t' = t_n$, such that for all $i$, $t_i$ and
$t_{i+1}$ intersect in at least one edge. A subcomplex
consisting of a single triangle is taken to be strongly connected.
It is easy to see that $A$ can be decomposed as a union of maximal
strongly connected components.
\end{definition}

\begin{lemma}\label{planar surfaces}
The interior of a strongly connected component $\sigma$ of $A$ is a planar surface. Let $N_v$ be the number of tetrahedrons that contain the vertex $v$. Let $N$ be the maximum $N_v$ over all vertices $v$ in the triangulation $\tau$. Then there are at most $N$ triangles in $\sigma$.
\end{lemma}
\begin{proof}
All triangles in the strongly connected component $\sigma$ of $A$ link the same vertex. Otherwise, there would be a common normal arc shared by a normal triangle of $F$ linking a vertex $v$ and a normal triangle of $F$ linking a distinct vertex $w$. This is a contradiction as normal arcs in a face, linking different vertices are not normally isotopic. So we can assume all triangles of $\sigma$ link the same vertex $v$.

To each normal triangle of $\sigma$ we can associate a height from the vertex $v$. Given a height $h$ from the vertex $v$, each tetrahedron has at most one triangle of $\sigma$ at height $h$. So, in all, there are at most $N_v$ triangles at height $h$. Triangles linking $v$ at different heights from $v$ do not intersect. So, as $\sigma$ is connected, all triangles of $\sigma$ can be assumed to be at the same height. Thus, there are at most $N_v$ triangles in $\sigma$.

Furthermore, the union of all possible normal triangles linking $v$ at the height $h$ (from $v$) gives a closed connected normal surface. Hence by remark \ref{vertexlinking sphere}, it is a vertex linking sphere $S(v)$. As $\sigma$ is composed of triangles at height $h$ from $v$, $\sigma \subset S(v)$. So the interior of $\sigma$ is a planar surface.
\end{proof}

\begin{lemma}\label{lemma1}
Let $|A'|$ denote the number of components of $A'$. Then,
$$ \chi(A') \geq 2|A'| - 4 Q $$
\end{lemma}

\begin{proof}

\begin{claim}\label{claim1}
$\chi(A') = 2|A'| - |\del A'|$.
\end{claim}
Let $\{A_n\}_n$ be the strongly connected components of $A$. Then, $int(A') = int(A) = \cup_n int(A_n)$. By lemma \ref{planar surfaces}, the interior of $A_n$ are planar surfaces. So, as $A'$ is a surface (not just a 2-complex), $A'$ is itself a disjoint union of connected planar surfaces.

A connected planar surface $S$ with $b$ boundary components has $\chi(S) = 2-b$. As $A'$ is a disjoint union of connected planar surfaces
 $\chi(A') = 2|A'| - |\del A'|$ as claimed.
\newline

\begin{definition}
Let $\sigma$ be a 2-complex. Define $w(e)$, weight of an edge $e
\in \sigma$, to be
$$ w(e) = \left\{
\begin{array}{ll}
1 & \mbox{if $e \subset \del \sigma$} \\
2 & \mbox{otherwise}
\end{array}\right.
$$
Define $w(\sigma)$, the weight of $\sigma$, to be the total weight
of all edges of $\sigma$. \\
For $\psi$ a 1-complex. Define $w(\psi)$ to be the total number of edges of $\psi$. \\
From the definition, it is clear that $w(\sigma) \geq w(\del \sigma)$.
\end{definition}

\begin{claim}\label{claim2}
$4Q \geq | \del A' |$
\end{claim}

The boundary of $B$ is a union of simple closed curves $\{ c_i \}$ that pairwise intersect in singular points. The set $\{c_i\}$ is in bijection with the connected components of $\del B'$. Therefore, $ w(\del B) = \Sigma w(c_i) \geq | \del B'|$.

As, $4Q = w(B) \geq w(\del B)$ and $\del B' = \del A'$, we get
the inequality $4Q \geq |\del A'| $ as claimed.
\newline

Using Claims \ref{claim1} and \ref{claim2}, we get the
desired relation
$$ \chi(A') \geq 2 |A'| - 4Q$$
\end{proof}

\begin{lemma}\label{lemma2}
For $B \neq \phi$,  $ \chi(B') \ge
\left\{
\begin{array}{ll}
 |\omega| - 3Q  & \mbox{when } |\omega| > 0 \\
 4 - 3Q         & \mbox{when } |\omega| = 0
\end{array}
\right.$
\end{lemma}

\begin{proof}
Let $B'_0 = \overline{N(\omega)}$. Let $B'_n$ be a union of $n$
quadrilaterals of $B'$ such that $B'_{n+1} \supset B'_n$. For a
quadrilateral $q$ such that $B'_{n+1} = B'_{n} \cup q$, $B'_n \cap
q$ is either a circle or a disjoint union of at most 4 arcs (or
points). So, $\chi(B'_n \cap q) \leq 4$, and we get
$$
\begin{array}{lll}
\chi(B'_{n+1}) &=& \chi(B'_n) + \chi(q) - \chi(B'_n \cap q) \\
 &\ge& ( \chi(B'_n) + 1 ) - 4 \\
\end{array}
$$
As $\overline{N(\omega)}$ is a disjoint union of discs,
$$ \chi(B'_0) = |\omega|$$
Therefore,
$$
\chi(B') = \chi(B'_Q) \ge
\left\{
\begin{array}{ll}
 |\omega| - 3Q  & \mbox{when } |\omega| > 0 \\
 1 - 3(Q-1)     & \mbox{when } |\omega| = 0
 \end{array}
 \right.
$$

\end{proof}

\begin{proof}[Proof of Theorem \ref{genus-quad}]
When $B= \phi$, the surface $F$ is composed of normal triangles and so, is a union of vertex linking spheres. Therefore as $A \neq \phi$, $\chi(F) = 2|A| \geq 2$.

Assume $B \neq \phi$, as $A' \cap B'$ is a disjoint union of circles. Therefore using lemmas \ref{lemma1} and \ref{lemma2} we get,
$$
\begin{array}{lll}
\chi(F)         &=&     \chi(A' \cup B') \\
                &=&     \chi(A') + \chi(B')\\
                &\ge&   \left\{
                        \begin{array}{ll}
                        (2|A'| - 4Q) + (|\omega| - 3Q)  &   \mbox{when $|\omega| > 0$}\\
                        (2|A'| - 4Q) + 4 - 3Q           &   \mbox{when $|\omega| = 0$}\\
                        \end{array}
                        \right.
\end{array}
$$
When $|A'|=0$ the surface is composed entirely of quadrilaterals,
so that $B \neq \phi$ and $|\omega| = 0$. So, we get the desired
relation
$$ \chi(F) \ge 2 - 7Q$$
\end{proof}

\section{An inequality relating the number of triangles and
quadrilaterals of normal surfaces}
In this section we prove an inequality relating the number of triangles and quadrilaterals of a
normal surface.

\begin{proof}[Proof of Theorem \ref{tria-quad}]

Let $F(v)$ be the union of normal triangles of $F$ linking the vertex $v$.
Let $\{ F(v)_n \}_n$ be the strongly connected components of $F(v)$.

Consider a graph $\Gamma$ whose vertices are the normal quadrilaterals
of $F$ and the non-empty $F(v)_n$. Call the former vertices
Q-vertices and the latter as S-vertices. Let there be an edge of
$\Gamma$ joining an S-vertex and a Q-vertex for every edge shared
by the corresponding $F(v)_n$ and the corresponding quadrilateral.
Let there also be edges of $\Gamma$ between Q-vertices
corresponding to every edge shared by the corresponding
quadrilaterals. For a quadrilateral in the pseudo triangulation we could have simple loops corresponding to quadrilaterals that are not injective on its edges.
\newline

As each quadrilateral has 4 sides, the degree of each Q-vertex is
4 and as there are no vertex-linking spheres, there are no
isolated S-vertices. As each edge of $\Gamma$ has at least one
vertex incident on a Q-vertex (no edges between S-vertices)
therefore, $4Q =$ total degree of Q vertices $\ge$ total degree of
S-vertices $\ge S$, where $Q$ and $S$ are the number of Q-vertices
and S-vertices respectively.
\newline

By claim \ref{planar surfaces} each $F(v)_n$ has at most $N$ triangles, where $N$ is the maximum
number of tetrahedrons linking a vertex in the triangulation of
$M$. Therefore, $N S \ge T$, where $T$ is the number of triangles
in $F$. Therefore, we get the required relation,
$$ T \le 4 N Q $$

\begin{remark}
This proof also holds for a normal surface $F$ embedded in a 3-complex, if we assumed that no component of $F$ was composed solely of normal triangles.
\end{remark}

\end{proof}

\acknowledgements{I would like to thank Siddhartha Gadgil for
useful discussions and advice. The CSIR-SPM Fellowship is
acknowledged for financial support.}

\bibliographystyle{amsplain}

\end{document}